\pdfoutput=1
\RequirePackage{ifpdf}
\ifpdf 
\documentclass[pdftex]{sigma}
\else
\documentclass{sigma}
\fi

\usepackage[cmtip,matrix,arrow]{xy}

\numberwithin{equation}{section}

\newtheorem{Theorem}{Theorem}[section]
\newtheorem*{Theorem*}{Theorem}
\newtheorem{Corollary}[Theorem]{Corollary}
\newtheorem{Lemma}[Theorem]{Lemma}
\newtheorem{Proposition}[Theorem]{Proposition}

\theoremstyle{definition}

\newtheorem{Remark}[Theorem]{Remark}

\newcommand{\Ra}{\Rightarrow}

\newcommand{\Co}{\mathcal{C}}

\newcommand{\F}{\mathcal{F}}

\newcommand{\R}{\mathbb{R}}
\newcommand{\C}{\mathbb{C}}

\newcommand{\Z}{\mathbb{Z}}
\newcommand{\Q}{\mathbb{Q}}

\newcommand{\on}{\operatorname}

\renewcommand{\ker}{ \on{ker}}

\newcommand\qu{/\kern-.7ex/} 

\newcommand{\hra}{\hookrightarrow}

\renewcommand{\d}{{\mbox{d}}}
\newcommand{\ol}{\overline}

\newcommand{\f}{\frac}

\newcommand{\eeq}{\end{eqnarray*}}
\newcommand{\beq}{\begin{eqnarray*}}

\newcommand{\pr}{\on{pr}}

\newcommand{\CC}{\mathsf{C}}

\renewcommand{\subset}{\subseteq}

\begin{document}

\allowdisplaybreaks

\newcommand{\arXivNumber}{2607.23859}

\renewcommand{\thefootnote}{}

\renewcommand{\PaperNumber}{107}

\FirstPageHeading

\ShortArticleName{A Note on the Cohomology of Pullback Lie Algebroids}

\ArticleName{A Note on the Cohomology of Pullback Lie Algebroids\footnote{This paper is a~contribution to the Special Issue on Interactions of Poisson Geometry, Lie Theory and Symmetry in honor of Rui Loja Fernandes for his 60th birthday. The~full collection is available at \href{https://sigma-journal.com/Fernandes.html}{https://sigma-journal.com/Fernandes.html}}}

\Author{Eckhard MEINRENKEN}

\AuthorNameForHeading{E.~Meinrenken}
\Address{Department of Mathematics, University of Toronto,\\
40 St. George Street, Toronto, Ontario M5S 2E4, Canada}
\Email{\mail{mein@math.toronto.edu}}
\URLaddress{\url{https://www.math.toronto.edu/mein/}}

\ArticleDates{Received July 28, 2026, in final form October 28, 2026; Published online November 04, 2026}

\Abstract{We give a direct differential-geometric proof of a result of Crainic on the cohomology of pullbacks of Lie algebroids under surjective submersions: If the fibers are connected and have vanishing cohomology through degree $n$, pullback with coefficients in any representation is an isomorphism through degree $n$ and is injective in degree $n+1$.	The proof reduces  to a  vanishing result for leafwise cohomology and uses an exhaustion-and-patching argument due to Buchdahl. No local triviality or finite-type assumption on the fibers is required. We also prove the analogous result for holomorphic Lie algebroids.}

\Keywords{Lie algebroids; Lie algebroid cohomology; foliated cohomology; van Est theory}

\Classification{58H05; 58A12; 57R30}

\renewcommand{\thefootnote}{\arabic{footnote}}
\setcounter{footnote}{0}

\section{Introduction}
	
Let $\pi\colon M\to N$ be a surjective submersion, and let
	$A\Ra  N$ be a Lie algebroid with a representation on a vector bundle $E\to N$.
	A theorem of Crainic~\cite[Theorem~2]{cra:dif}, used in his proof
	of the van Est theorem for Lie groupoids, asserts that
	if the fibers of $\pi$ are connected and have vanishing
	cohomology in degrees $1,\dots,n$, then the pullback map
	\[
	\pi^*\colon \ H^k(A,E)\longrightarrow H^k\bigl(\pi^!A,\pi^*E\bigr)
	\]
	is an isomorphism for $k\leq n$ and is injective in
	degree $n+1$. As a special case, taking $A$ to be the zero Lie algebroid with representation on the trivial line bundle, the foliated (or leafwise) de Rham cohomology $H_\pi^k(M)$ vanishes in degrees $1\le k\le n$.  His argument involves a sheaf-theoretic acyclicity criterion of Bernstein and Lunts~\cite{ber:equ}. The criterion relies on substantial machinery, and verifying its hypotheses for general smooth submersions is nontrivial.
	
The purpose of this note is to present a direct differential-geometric
	proof. The main ingredient is an exhaustion-and-patching argument due
	to Buchdahl \cite{buc:rel}, which we adapt to the present setting.
	Our proof uses only standard tools from differential topology and de
	Rham theory, at the level of Bott and Tu \cite{bo:di}. Besides making
	the argument self-contained, this approach displays  the
	main difficulties: a general surjective submersion need not be a locally trivial fibration, and its fibers need not admit finite good covers.
	
	We also prove the holomorphic counterpart of Crainic's theorem: Suppose $\pi\colon M\to N$ is a~surjective  holomorphic submersion, and
	$A\Ra N$ is a holomorphic Lie algebroid with a representation
	on a holomorphic vector bundle $E$. If the fibers are connected and have vanishing ordinary de Rham cohomology in degrees $1,\dots,n$,
	then the pullback map in holomorphic Lie algebroid cohomology
	\[
	\pi^*\colon \ H^k_{\on{hol}}(A,E)\longrightarrow H^k_{\on{hol}}\bigl(\pi^!A,\pi^*E\bigr)
	\]
	is an isomorphism for $k\leq n$ and is injective in
	degree $n+1$. The proof relies on the description of holomorphic Lie
	algebroid cohomology in terms of complex Lie algebroids
	due to Laurent-Gengoux, Sti\'{e}non, and Xu~\cite{lau:hol}.
	
\section{Pullback Lie algebroids}
	All manifolds considered in this note are paracompact and Hausdorff.
	Let $A\Ra N$ be a Lie algebroid with anchor $\varrho\colon A\to TN$. The Lie algebroid cohomology of $A$ with coefficients in an $A$-representation
	$E\to N$ is denoted $H^\bullet(A,E)$. Given a smooth map $f\colon M\to N$ transverse to the anchor, one defines
	the pullback Lie algebroid $f^!A\Ra M$ (using standard notation as in \cite[Section 13.4]{cra:lec}) as the fiber product
	\[ f^!A=TM\times_{TN} A,\]
	with Lie algebroid structure as a subalgebroid of $TM\times A$. Projection to the second factor gives a Lie algebroid morphism $f^!A\to A$.
	Given a representation of $A$ on a vector bundle $E\to N$, one obtains a representation of
	$f^!A$ on the (usual) pullback bundle $f^*E$.
	
	In this note, we only consider the case that $f$ is given by a surjective submersion
	\[\pi\colon \ M\to N.\]
	In this case, transversality is automatic, and the Lie algebroid morphism $\pi^!A\to A$ is fiberwise surjective,  with kernel
	$T_\pi M=\ker(T\pi)$ the tangent bundle to the foliation defined by $\pi$. Sections of $\wedge T_\pi^* M$ are called foliated (or leafwise)
	differential forms; we use the notation $\d_\pi$ for the Lie algebroid differential in order to avoid confusion with the de Rham differential.
	
	\begin{Theorem}[Crainic \cite{cra:dif}]\label{th:main}
		Let $A\Ra N$ be a Lie algebroid, with a representation on a vector bundle $E\to N$. Suppose 	
		$\pi\colon M\to N$ is a surjective submersion with connected fibers, such that the cohomology of the fibers vanishes in degrees $1\le k\le n$.
		Then the pullback map
		\[ \pi^*\colon  \ H^\bullet(A,E)\to H^\bullet(\pi^!A,\pi^*E)\]
		is an isomorphism in degrees $0\le k\le n$, and is injective in degree $n+1$. 	
	\end{Theorem}

	As a special case of Theorem~\ref{th:main}, taking $A=0_N$ to be the zero Lie algebroid with the trivial representation on $E=N\times \R$, we have that
	$\pi^!A=T_\pi M$. The corresponding Lie algebroid cohomology is identified with the leafwise (foliated) cohomology,
	\[ H^\bullet(T_\pi M)=H^\bullet_\pi(M).\]
	Consequently, we have the following.

	\begin{Proposition}\label{prop:keycorollary}
		Let $\pi\colon M\to N$ be a surjective submersion with connected fibers, such that the cohomology of the fibers vanishes in degrees $1\le k\le n$.
		Then the foliated cohomology $H^\bullet_\pi(M)$ vanishes in degrees  $1\le k\le n$.
	\end{Proposition}
	
	\begin{Remark}
		The corresponding statement is false for general foliations $\F$: In particular, $H^1_\F(M)$ may be non-zero even if all leaves are simply connected. For example, let  $M=\R^2/\Z^2$ be the 2-torus, and let
		\[ \theta_1,\theta_2\in \Omega^1(M)\] be closed 1-forms pulling back to
		\[ \d x,\, \d y+\mu \d x\in \Omega^1\bigl(\R^2\bigr),\]
		respectively, where $\mu\not\in\Q$.
		The Kronecker foliation $\mathcal F$, given in terms of its tangent bundle by
		\[ T_\F M=\ker\theta_2\subseteq TM,\]
		has all leaves diffeomorphic to
		$\mathbb R$. Hence $H^1(L)=0$ for every leaf $L$.  The restriction of $\theta_1$ to leaves is a foliated form $\alpha\in \Omega^1_\F(M)$. Its class in  $H^1_\F(M)$ is non-zero: If $\alpha=\d_\F f$ with $f\in C^\infty(M)$, then $\d f-\theta_1$ vanishes on  $T_\F M$ and hence is a  function times $\theta_2$. But then
		\[ \int \theta_1\wedge\theta_2=\int (\d f)\wedge\theta_2=\int \d(f\theta_2)=0,\]
		which is a contradiction since $\theta_1\wedge\theta_2$ is a volume form.
	\end{Remark}
	
	\begin{Remark}
The vanishing result of Proposition~\ref{prop:keycorollary}
		is also used by Sparano and Vitagliano~\cite{spa:def} in their work on Morita invariance
		of deformation cohomology of Lie algebroids.
	\end{Remark}
	
Following \cite{cra:dif}, we will prove Proposition~\ref{prop:keycorollary} as a first step towards  Theorem~\ref{th:main}. To highlight the difficulty, let us point out that the proposition is greatly simplified when $\pi$ is a locally trivial fibration whose fibers admit finite good covers. Indeed, one can then reduce to the case of direct products, and use the following fact:
	
	\begin{Proposition}[products] \label{prop:product}
		Suppose $M=Q\times N$ 	is a product of two manifolds, where $Q$ admits a finite good cover, and let $\pi\colon M\to N$ be the projection to the second factor.
		Then
		\[ H_\pi^\bullet(M)=H^\bullet(Q)\otimes C^\infty(N).\]
	\end{Proposition}
	\begin{proof}
		We want to show that the inclusion
\[
\iota\colon\	\Omega^\bullet(Q)\otimes C^\infty(N)\to \Omega^\bullet_\pi(Q\times N)
\]
		induces an isomorphism in cohomology.
		If $Q$ is smoothly contractible, this follows by  the Poincar\'{e} lemma on $Q$, treating the dependence on $N$ as parameters. In the general case, pick a~finite good cover of~$Q$, and apply the usual Mayer--Vietoris argument \cite[Chapter~I, \S5]{bo:di}.
	\end{proof}

\section{The Buchdahl argument}\label{app:buch}

Throughout the following discussion, $\pi\colon M\to N$ denotes a surjective submersion.

\begin{Proposition}[Buchdahl \cite{buc:rel}]\label{prop:buch}
		Suppose that for some $n\ge 1$,
		\[H^n\bigl(\pi^{-1}(x)\bigr)=0,\qquad  H^{n-1}\bigl(\pi^{-1}(x)\bigr)=0
		\]
		for all $x\in N$. $($For $n=1$, replace this with the assumption that $H^1\bigl(\pi^{-1}(x)\bigr)=0$ for all $x\in N$ and  the fibers are connected.$)$
		Then
		\[H^n_\pi(M)=0.\]
\end{Proposition}

We shall give a proof in Section~\ref{subsec:buch} below, using a modified version of the argument in \cite{buc:rel} and elementary results from de Rham theory.  Clearly, Proposition~\ref{prop:keycorollary} follows from Proposition~\ref{prop:buch}.
	
	\subsection{Necessity of hypotheses}
	The condition that $H^{n-1}\bigl(\pi^{-1}(x)\bigr)=0$ cannot be omitted, as the following example shows.
	Consider the submersion
	\[ \pi\colon \ M=\R^3\setminus\{(0,0,0)\}\to \R ,\qquad (x,y,z)\mapsto x.\]
	Here
	\[ H^2\bigl(\pi^{-1}(x)\bigr)=0,\qquad H^1\bigl(\pi^{-1}(x)\bigr)=\begin{cases}
		0 & x\neq 0,\\
		\R & x=0.
	\end{cases}\]
	Any form $\alpha\in \Omega^2_\pi(M)$ is $\d_\pi$-closed, for dimension reasons. Fiberwise integration of $\alpha$
	over the unit disk $D(x)\subset \pi^{-1}(x)$,  for $x\neq 0$,
	defines a smooth function $f_\alpha\in C^\infty(\R\setminus \{0\})$,
	\[ f_\alpha(x)=\int_{D(x)} \alpha.\]
	If $\alpha=\d_\pi\beta$, we may use Stokes' theorem to write
	\[ f_\alpha(x)=\int_{S(x)} \beta,\]
	where $S(x)\subset \pi^{-1}(x)$ is the unit circle. But this integral is also defined for $x=0$, which shows that
	$f_\alpha$ extends smoothly across $x=0$. We have hence constructed a map
	\begin{equation}\label{eq:h2pimap}
		H^2_\pi(M)\to C^\infty(\R\setminus \{0\})/C^\infty(\R)
	\end{equation}
	to smooth functions on the punctured real line, modulo restrictions of functions on  the real line. This map is non-zero:
	for example, for the foliated 2-form
	\[ \alpha=\f{\d y\wedge \d z}{x^2+y^2+z^2}\in \Omega^2_\pi\bigl(\R^3\setminus \{0\}\bigr)\]
	the function $f_\alpha$
	diverges at $x=0$ and so does not extend.
	
\begin{Remark}
		Using a Mayer--Vietoris argument, one can show that \eqref{eq:h2pimap} is an isomorphism.
\end{Remark}
	
\subsection{Preparatory lemmas}\label{subsec:prep}
	Recall \cite[Chapter I, \S 5]{bo:di} that a \emph{good cover}  of a manifold $M$ is an open
	cover such that every non-empty finite intersection of members of the
	cover is contractible. It is a standard fact \cite[Chapter~I, \S1, Corollary II]{gr:co1} that every manifold~$M$ admits
	a countable good cover. A manifold is of \emph{finite type} if it admits a finite good cover. This then implies that the de Rham cohomology of $M$ is finite dimensional.
	Given a manifold $M$, not necessarily of finite type, there exists an exhaustion by open subsets $U_1\subset U_2\subset \cdots$ with compact closures, such that
	\[ \bigcup_a U_a=M, \qquad \ol{U}_a\subset U_{a+1},\]
	and each $U_a$ is of finite type. (Just take a good open cover by open subsets with compact closures, and take $U_a$ to
	consist of suitable increasing finite unions of that cover.) Given such an exhaustion, we have the following.
	
	\begin{Lemma}\label{lem:3.3}
		Suppose $H^k(M)=0$ for some $k$. Then there exists, for every $a$, some $b>a$ such that
		the map  $H^k(U_{b})\to H^k(U_{a})$ is the zero map.
	\end{Lemma}
\begin{proof}
Let $H^\bullet_{\on{comp}}(M,o_M)$ be the cohomology with compact support, with values in the orientation bundle~$o_M$.
		Poincar\'{e} duality \cite[Chapter~V, \S4, Theorem~I]{gr:co1}
		states that for every manifold~$M$ (not necessarily of finite type), the natural map
		\[ H^\bullet(M)\to 	\bigl(H_{\on{comp}}^{\dim M-\bullet}(M,o_M)\bigr)^*\]
		is an isomorphism.

The statement of the lemma is hence equivalent to saying that
		if	$\smash{H^{\ell}_{\on{comp}}(M,o_M)}=0$ for some $\ell$, then there exists, for every given $a$, some $b>a$ such that
		the map  $\smash{H^\ell_{\on{comp}}(U_{a},o_M)}\to \smash{H^\ell_{\on{comp}}(U_{b},o_M)}$ is the zero map. Since
		$U_a$ is of finite type, the cohomology group \smash{$H^\ell_{\on{comp}}(U_{a},o_M)$} is finite dimensional. We may therefore choose closed forms
		\[ \alpha_1,\dots,\alpha_r\in  \Omega^\ell_{\on{comp}}(M,o_M)
		\]
		supported in $U_a$, such that their cohomology classes form a basis. Since \smash{$H^{\ell}_{\on{comp}}(M,o_M)=0$},
		these forms admit primitives
		\[\beta_1,\dots,\beta_r
		\in  \Omega^{\ell-1}_{\on{comp}}(M,o_M).\]
		These primitives need not be supported in $U_a$, but we can choose  $b>a$ so that
		$U_b$ contains the support of these primitives. It then follows that $\alpha_i|_{U_b}$ are exact.
		Hence
		$\smash{H^{\ell}_{\on{comp}}(U_a,o_M)}\to  \smash{H^{\ell}_{\on{comp}}(U_b,o_M)}$ is the zero map, and so is the dual map.
	\end{proof}
	
	We shall need the analogous results for submersions  $\pi\colon M\to N$.

	\begin{Lemma}[fiberwise good covers] \label{lem:fiberwisegood}
		Given a submersion $\pi\colon M\to N$, we may choose a locally finite cover of $M$ by open sets, with compact closures, such that for every $x\in N$ the resulting cover of the fiber $\pi^{-1}(x)$ obtained by intersections is a good cover.
	\end{Lemma}
	\begin{proof}
		Recall that a subset $W$ of a Riemannian manifold $(Q,g)$ is geodesically  convex if any two points of $W$  are joined by a
		unique length minimizing geodesic in $Q$, contained in $W$.
		The intersection of geodesically convex subsets is again geodesically convex. It is a standard fact that for every $p\in Q$,
		there exists $\epsilon>0$ such that the open ball
		$B_\epsilon(p;g)$ is geodesically convex.
		For $\epsilon$ sufficiently small, this property
		is stable under small perturbations of the metric
		$g$ and of the center $p$.
		
		Choose a Euclidean inner product on $T_\pi M=\ker(T\pi)$. This determines a family
		of Riemannian metrics on the fibers of $\pi$,
		depending smoothly on the base point in $N$.
		Call an open subset $U\subseteq M$
		\emph{fiberwise geodesically convex} if it has the property
		\begin{equation*}\tag{$*$}
			U\cap\pi^{-1}(x)\text{ is geodesically convex
				for every }x\in N.
		\end{equation*}
		Every point $p\in M$ has arbitrarily small fiberwise geodesically convex neighborhoods with compact closure.
		Indeed, choose a local section $s\colon V_0\to M$
		through $p$. After shrinking to an open neighborhood
		$V$ with $\ol V\subset V_0$, and choosing
		$\epsilon>0$ sufficiently small, the fiberwise
		$\epsilon$-balls centered at $s(x)$, for $x\in V$,
		are geodesically convex.
		If $\ol V$ is compact, then this union has compact closure in $M$.
		
To construct a locally finite cover, we use a standard
		exhaustion argument (see, e.g., \cite[Exercise~4.4]{fer:diff}).
		Choose a nested sequence of open subsets
		$O_1\subset O_2\subset\cdots$ of $M$, with compact
		closures, such that
		\[
		\bigcup_a O_a=M,\qquad \ol O_a\subset O_{a+1},
		\]
		and put $O_0=O_{-1}=\varnothing$.
		For each $a\ge 1$, choose a finite cover of
		the compact set
		\[
		\ol O_a\setminus O_{a-1},
		\]
		by  fiberwise geodesically convex open sets, with compact closures contained in
		\[
		O_{a+1}\setminus\ol O_{a-2}.
		\]
		The union of these finite collections is the
		desired cover of $M$. It is locally finite, since the
		open sets $O_{a+1}\setminus\ol O_{a-2}$
		form a locally finite family.
	\end{proof}
	
\begin{Lemma}[fiberwise exhaustions]\label{lem:3.4}
Let $\pi\colon M\to N$ be a submersion. Then there exists a~nested sequence of open subsets $U_1\subset U_2\subset \cdots $ of $M$, with compact closures,
		such that
		\[ \bigcup_a U_a=M, \qquad \ol{U}_a\subset U_{a+1},\]
		and such that each $\pi^{-1}(x)\cap U_a$ is of finite type for every $x\in N$.
	\end{Lemma}
	\begin{proof} As in the construction of the exhaustion preceding Lemma~\ref{lem:3.3}, we may build the $U_a$ by taking increasing finite unions of members of a fiberwise good cover.
	\end{proof}
	
\begin{Lemma}[product neighborhoods]\label{lem:3.5}
		Given $x\in N$ and any open subset $U\subset M$ with compact closure, there exists an open neighborhood $V$ of $x$ and an embedding
		\[ 	\Phi\colon \ \bigl (U\cap \pi^{-1}(x)\bigr)\times V\to M\]
		intertwining projection to the second factor with the map $\pi$.
	\end{Lemma}
	\begin{proof}
		Such a neighborhood may be constructed using Euler-like vector fields \cite{bur:spl}.
		The map $\pi$ induces a map of normal bundles 	$ \nu\bigl(M,\pi^{-1}(x)\bigr)\to \nu(N,\{x\})=T_xN$,
		which is a fiberwise isomorphism (since $\pi$ is a submersion). One hence has a trivialization
		\[
		\nu(M,\pi^{-1}(x))\cong \pi^{-1}(x)\times T_xN.
		\]
		Pick an Euler-like vector field on $N$ with respect to $\{x\}$. It admits a lift to an Euler-like vector field on $M$ with respect to $\pi^{-1}(x)$. One obtains open neighborhoods $O_M\subset \nu(M,\pi^{-1}(x))$ of~$\pi^{-1}(x)$
		and $O_N\subset \nu(N,\{x\})$ of $\{x\}$,  with tubular neighborhood embeddings
		\[ \phi_M\colon \ O_M\to M,\qquad \phi_N\colon \ O_N\to N\]
		which intertwine the submersion  $O_M\to O_N$ (given by restriction of $\nu\bigl(M,\pi^{-1}(x)\bigr)\to  T_xN$)
		with the given submersion $\pi$. The subset
		$O_M$ may be strictly smaller than	$ \pi^{-1}(x)\times O_N$ in general. However, since $\ol{U}\cap \pi^{-1}(x)$ is compact, we
		may choose an open neighborhood $V$ of $x$ such that $\bigl(\ol{U}\cap \pi^{-1}(x)\bigr)\times \phi_N^{-1}(V)\subset O_M$. The map $\Phi$ is obtained by restriction of $\phi_M\circ \bigl(\on{id}\times \phi_N^{-1}\bigr)$.
	\end{proof}
	
\begin{Lemma}\label{lem:3.6}
		Suppose $H^k\bigl(\pi^{-1}(x)\bigr)=0$ for all $x\in N$,
		for some $k\ge 1$.
		Then, for any exhaustion $U_1\subseteq U_2\subseteq\cdots$ as in
		Lemma~{\rm \ref{lem:3.4}}, and every $a$, there exists
		$b>a$ such that the restriction map
		\[
		H^k_\pi(U_b)\to H^k_\pi(U_a)
		\]
		is zero.
	\end{Lemma}
	
\begin{proof}
Given $x\in N$, fix  $c>a+1$ so that 	
		\[ H^k\bigl(U_{c}\cap \pi^{-1}(x)\bigr)\to H^k\bigl(U_{a+1}\cap \pi^{-1}(x)\bigr)\]
		is the zero map.
		Choose a product open neighborhood $\Phi\bigl((U_c\cap \pi^{-1}(x))\times V\bigr)$ as in
		the previous lemma.
		Shrinking $V$ if necessary, the product neighborhood  $\Phi\bigl(\bigl(U_{a+1}\cap \pi^{-1}(x)\bigr)\times V\bigr)$ contains
		$\ol{U}_{a}\cap \pi^{-1}(V)$, and taking $b$ sufficiently large the neighborhood $\Phi\bigl(\bigl(U_c\cap \pi^{-1}(x)\bigr)\times V\bigr)$
		is contained in $U_b$. This gives inclusions
		\[ U_{a}\cap \pi^{-1}(V)\hra  \Phi\bigl(\bigl(U_{a+1}\cap \pi^{-1}(x)\bigr)\times V\bigr)\hra \Phi\bigl(\bigl(U_c\cap \pi^{-1}(x)\bigr)\times V\bigr)
		\hra U_b\cap \pi^{-1}(V)
		.\]
		The map
		\[ H^k_\pi \bigl(\Phi\bigl(\bigl(U_c\cap \pi^{-1}(x)\bigr)\times V\bigr)\bigr)\to  H^k_\pi \bigl(\Phi\bigl(\bigl(U_{a+1}\cap \pi^{-1}(x)\bigr)\times V\bigr)\bigr)\]
		induced by the middle inclusion is the zero map, since
		\[ H^k_\pi \bigl(\Phi\bigl(\bigl(U\cap \pi^{-1}(x)\bigr)\times V\bigr)\bigr) =H^k\bigl(U\cap \pi^{-1}(x)\bigr)\otimes C^\infty(V)\]
		by Proposition~\ref{prop:product}. It follows that the map
		\[ H^k_\pi\bigl(U_{b}\cap \pi^{-1}(V)\bigr)\to H^k_\pi\bigl(U_a\cap \pi^{-1}(V)\bigr)\]
		is the zero map.  The base manifold $N$ may be covered by open subsets $V$ of this sort. Since
		$U_a$ has compact closure, there is a finite collection $\{V_i\}$ of such open subsets whose inverse images
		$\bigl\{\pi^{-1}(V_i)\bigr\}$ cover $\ol{U}_a$. Choose $b_i$ such that the maps
		\[ H^k_\pi\bigl(U_{b_i}\cap \pi^{-1}(V_i)\bigr)\to H^k_\pi\bigl(U_a\cap \pi^{-1}(V_i)\bigr)\]
		are all zero, and let $b$ be their maximum. Given $\alpha\in \Omega^k_\pi(U_b)$ with $\d_\pi\alpha=0$,
		the restriction to $U_{a}\cap \pi^{-1}(V_i)$ admits a primitive $\beta_i$.
		Choose  $\chi_i\in C^\infty(N)$, supported in
		$V_i$, such that $\sum_i\chi_i=1$ on $\pi\bigl(\overline{U_a}\bigr)$.
		Then
		\[
		\beta=\sum_i(\pi^*\chi_i)\beta_i
		\]
		satisfies  $\d_\pi\beta=\alpha$ on $U_a$.
	\end{proof}
	
\subsection{Lie algebroid cohomology and exhaustions}
	
	We first record a general observation about Lie algebroid cohomology via exhaustions.
	Suppose
	\[ U_1\subset U_2\subset \cdots \]
	are  open subsets of $M$, with compact closures, satisfying
	\[ \ol{U}_a\subset U_{a+1},\qquad\bigcup_a U_a=M.\]

	\begin{Lemma}\label{lem:liealgebroid}
		Let $A\Ra M$ be a Lie algebroid with a
		representation on a vector bundle $E\to M$. Suppose that for some $n>1$,
		the restriction maps
		\[
		H^k(A|_{U_{a+1}},E|_{U_{a+1}})
		\longrightarrow
		H^k(A|_{U_a},E|_{U_a})
		\]
		are zero for $k=n-1,n$ and all $a=1,2,\dots$. Then
		$H^n(A,E)=0$.
	\end{Lemma}
	\begin{proof}
		For any open subset $U\subseteq M$, write
		\[
		\mathcal C^k(U)=\Gamma\bigl(U,\wedge^k A^*\otimes E\bigr)
		\]
		for the space of Lie algebroid cochains over $U$,
		with differential~$\d_A$.
		
Choose functions $\lambda_a\!\in\! C^\infty(M)$
		supported in $U_{a+1}$, with
		$\lambda_a|_{U_a}\!=\!1$.
		To show that ${H^n(A,E)\!=\!0}$, let
		$\alpha\in\mathcal C^n(M)$ be a cocycle.
		For every $a$, the restriction of $\alpha$
		to $U_{a+2}$ is exact. Choose
		\[
		\beta'_a\in\mathcal C^{n-1}(U_{a+2}),
		\qquad
		\d_A\beta'_a=\alpha|_{U_{a+2}}.
		\]
		We shall modify these primitives inductively. Put
		$\beta_1=\beta'_1$, and suppose that
		\[
		\beta_a\in\Co^{n-1}(U_{a+2})
		\]
		has been constructed.  The difference $\beta'_{a+1}|_{U_{a+2}}-\beta_a\in \Co^{n-1}(U_{a+2})$ is closed, and becomes exact after further restriction to $U_{a+1}$. Let 	
		\[ \gamma_a\in \Co^{n-2}(U_{a+1})\]
		be a primitive, so that
		\[ \d_A \gamma_a= \beta'_{a+1}|_{U_{a+1}}-\beta_a|_{U_{a+1}}.\]
		Since $\lambda_a\gamma_a$ has compact support in $U_{a+1}$, it
		extends by zero to $U_{a+3}$. Define
		\[ \beta_{a+1}=\beta'_{a+1}-\d_A(\lambda_a \gamma_a).\]
		Then $\beta_{a+1}$ agrees with $\beta_a$ on $U_a$, and is still a primitive of $\alpha|_{U_{a+3}}$.
		Consequently, the restrictions~$\beta_a|_{U_a}$
		patch together to define a global cochain
		$\beta\in\Co^{n-1}(M)$ satisfying $\d_A\beta=\alpha$.
	\end{proof}
	
	\begin{Remark}	Note that the proof only uses cutoff functions and the gluing
		property of sections. It therefore applies more generally
		to complexes of sheaves of vector spaces whose terms
		are modules over smooth functions.
	\end{Remark}

	\subsection{Proof of Proposition~\ref{prop:buch}} \label{subsec:buch}
	
	Consider first the case $n=1$, so that all fibers are connected and have vanishing first cohomology.
	Choose an open cover $\{V_i\}$ of $N$ such that $\pi$ admits
	sections
	\[s_i\colon \ V_i\to \pi^{-1}(V_i)\subset  M,\]
	and let $\{\chi_i\}$ be a partition of unity subordinate to this cover. Given a $\d_\pi$-closed form
	$\alpha\in\Omega^1_\pi(M)$  we obtain unique functions $f_i\colon \pi^{-1}(V_i)\to \R$ such that $f_i\circ s_i=0$ and
	the restriction of $f_i$ to each fiber $\pi^{-1}(x)$, $x\in V_i$, is a primitive of the restriction of  $\alpha$. The function $f_i$ is smooth. (To~see this near a given point $p\in \pi^{-1}(x)$, choose a path in $\pi^{-1}(x)$  from $p_0=s_i(x)$ to $p$.
	Subdivide the path into finitely many segments joining successive
	points
	\[
	p_0,p_1,\dots,p_N=p,
	\]
	each of which is contained in a submersion chart. It then follows inductively that $f_i$ is smooth near $p_1$, hence
	smooth near $p_2$, and so on.)  The function
	\[
	f=\sum_i (\pi^*\chi_i)f_i
	\]
	satisfies $\d_\pi f=\alpha$. This proves $H^1_\pi(M)=0$.
	
	For the case $n>1$, start with an exhaustion
	$U_1\subset U_2\subset\cdots$ as in Lemma
	\ref{lem:3.4}. By Lemma~\ref{lem:3.6}, after passing to a
	subsequence we may assume that the restriction maps
	\[
	H^k_\pi(U_{a+1})\to H^k_\pi(U_a)
	\]
	are zero for $k=n-1,n$ and $a=1,2,\dots$.
	The result now follows from Lemma~\ref{lem:liealgebroid},
	applied to the Lie algebroid $A=T_\pi M$
	with its trivial representation.
	
	\subsection{Coefficients}
	As already mentioned, Proposition~\ref{prop:keycorollary}
	is a direct consequence of Proposition~\ref{prop:buch}. In fact, we obtain a more general statement by considering coefficients.
	Let
	\[ E\to N\]
	be a given vector bundle. Then $\pi^*E$ is fiberwise trivial, and has a canonical $T_\pi M$-representation,
	defining $H^\bullet_\pi(M,\pi^*E)$.
	
	\begin{Corollary}\label{cor:buchcorollary}
		Under the assumptions of Proposition~{\rm \ref{prop:keycorollary}}, let $E\to N$ be a vector
		bundle. Then
		\[
		H^k_\pi(M,\pi^*E)=0,
		\qquad 1\leq k\leq n,
		\]
		and
		\[
		H^0_\pi(M,\pi^*E)=\Gamma(E).
		\]
	\end{Corollary}

\begin{proof}The degree-zero statement follows from connectedness of the fibers:
		a smooth section of $\pi^*E$ that is constant along the fibers
		descends to a section of $E$. For the positive-degree
		statement, choose an open cover of~$N$ over which~$E$ is trivial.
		The assertion holds over the inverse image of each member of this cover by
		Proposition~\ref{prop:buch}, applied componentwise. The local primitives may be
		patched using a partition of unity on~$N$.
\end{proof}
	
\section{The pullback theorem and its extensions}
	We now use the vanishing result from the previous section to prove
	Theorem~\ref{th:main}. After choosing a splitting of the pullback Lie algebroid,
	the proof is a standard zig-zag argument. We then discuss some special cases and
	extensions.
	
	\subsection{Proof of Theorem~\ref{th:main}}
	We assume that the fibers of $\pi\colon M\to N$ are connected and have vanishing cohomology in degrees $1\le k\le n$. Let
	$A\Ra N$ be a Lie algebroid, and $E\to N$ a representation of $A$. The Lie algebroid complex of $A$ with coefficients in $E$ is denoted
	\[ \CC^\bullet(A,E)=\Gamma(\wedge^\bullet A^*\otimes E)\]
	with the Chevalley--Eilenberg differential $\d_A$.
	Similarly, we have the Lie algebroid complex of~$\pi^!A$ with coefficients in~$\pi^*E$. The pullback map
\[
		\pi^*\colon \ \CC^\bullet(A,E)\to \CC^\bullet\bigl(\pi^!A,\pi^*E\bigr)
\]
	is a cochain map. After choosing a splitting of the exact sequence of vector bundles
	\[ 0\to T_\pi M\to \pi^!A\to \pi^*A\to 0,\]
	we have $\CC^k\bigl(\pi^!A,\pi^*E\bigr)=\bigoplus_{p+q=k}\CC^{p,q}$ with
	\[ \CC^{p,q}=\Omega^q_\pi(M,\pi^*(\wedge^p A^*\otimes E)).\]
	Using the Chevalley--Eilenberg formula for the Lie algebroid differential $\d_{\pi^!A}$ and the decomposition
	$\pi^!A=T_\pi M\oplus \pi^*A$, we see that $\d_{\pi^!A}$ is the sum of
	\[ \d_\pi\colon \ \CC^{p,q}\to \CC^{p,q+1},\]
	and two other summands of bidegrees $(1,0)$ and $(2,-1)$.
	By Corollary \ref{cor:buchcorollary}, the cohomology of~$\CC^{p,\bullet}$ vanishes in degrees $1\le q\le n$, and is given by $\Gamma(\wedge^p A^*\otimes E)$ for $q=0$.
	
	We now use the standard zig-zag argument: To show that the map
	\begin{equation}\label{eq:pistarcoh}
		\pi^*\colon  \ H^k(A,E)\to H^k\bigl(\pi^!A,\pi^*E\bigr)
	\end{equation}
	is injective for $k\le n+1$, let $\alpha\in \CC^k(A,E)$ be a $\d_A$-cocycle such that $\pi^*\alpha$ is exact. We need to show that $\alpha$ is exact. Write $\pi^*\alpha=\d_{\pi^! A}\beta$ for $\beta\in \CC^{k-1}\bigl(\pi^!A,\pi^*E\bigr)$.
	It decomposes as
	\[ \beta=\beta_q+\dots+\beta_0\]
	with $\beta_i\in \CC^{k-1-i,i}$. Then $\d_\pi \beta_q=0$. If $q>0$, it follows that $\beta_q=\d_\pi \gamma$
	for some element $\gamma\in \CC^{k-1-q,q-1}$.
	Replacing $\beta$ with $\beta-\d_{\pi^!A}\gamma$, we have arranged that $\beta_q=0$, but $\d_\pi \beta_{q-1}=0$.
	Proceeding in this manner,
	we eventually arrive at an element  $\beta\in \CC^{k-1,0}$ with $\d_{\pi^!A}\beta=\pi^*\alpha$. This means in particular $\d_\pi\beta=0$, hence
	$\beta=\pi^*\gamma$ for some $\gamma\in \CC^{k-1}(A,E)$. It follows that
	\[ \pi^*\d_A\gamma=\d_{\pi^!A}\beta=\pi^*\alpha,\]
	hence $\d_A\gamma=\alpha$. This shows $\alpha$ is exact, as required.
	Similarly, to see that the map \eqref{eq:pistarcoh} is surjective for $k\le n$, let $\sigma\in  \CC^{k}\bigl(\pi^!A,\pi^*E\bigr)$ be a cocycle
	representing a given cohomology class. Write
	$\sigma=\sigma_q+\dots+\sigma_0$ with $\sigma_i\in \CC^{k-i,i}$. Then $\d_\pi\sigma_q=0$. If $q>0$, it follows that
	$\sigma_q=\d_\pi\tau$ for some $\tau\in \CC^{k-q,q-1}$. Replacing $\sigma$ with $\sigma-\d_{\pi^!A}\tau$, we have arranged $\sigma_q=0$.
	Proceeding in this manner, we see that the given cohomology class may be represented by an element $\sigma$ in $\CC^{k,0}$. Since
	$\d_\pi\sigma=0$, this element is of the form $\sigma=\pi^*\alpha$ for some $\alpha\in \CC^k(A,E)$, with $\pi^*\d_A\alpha
	=\d_{\pi^!A}\pi^*\alpha=\d_{\pi^!A}\sigma=0$, hence $\d_A\alpha=0$.
	\qed
	
\subsection{Special cases and remarks}
	
\subsubsection[The case A=TN]{The case $\boldsymbol{A=TN}$}
We have already mentioned that for the zero Lie algebroid $A=0_N$, with $E=N\times \R$, Theorem~\ref{th:main} reduces to Proposition~\ref{prop:keycorollary}. Another noteworthy case is $A=TN$, so that $\pi^!A=TM$. An $A$\nobreakdash-representation on a vector bundle $E\to N$ is the same as a flat connection on $E$, and $H^\bullet(A,E)$ is the de Rham cohomology with coefficients in the flat vector bundle $E$.
	Theorem~\ref{th:main}  says that if $\pi\colon M\to N$ has connected fibers with vanishing cohomology in degrees $1\le k\le n$, then the pullback map
	\[ H^k(N,E)\to H^k(M,\pi^*E)\]
	is an isomorphism in degrees $0\le k\le n$, and is injective for $k=n+1$.
	
	\subsubsection{Complex Lie algebroids}
	Throughout our discussion, it was tacitly assumed that $A$ is a \emph{real} Lie algebroid. The argument also applies to \emph{complex} Lie algebroids (with anchor map  taking values in the complexified tangent bundle $T^\C N$), with representations on complex vector bundles $E\to N$.
	Here the pullback $\pi^!A$ is defined as the fiber product
	$T^\C M\times_{T^\C N} A$, and comes with a Lie algebroid morphism to $A$ with kernel $T_\pi^\C M$. The statement and proof of Theorem~\ref{th:main} extend to this setting, keeping in mind that $\CC^\bullet (A,E)$ and $\CC^\bullet\bigl(\pi^!A,\pi^*E\bigr)$ are now complex vector spaces.

\subsubsection{Non-Hausdorff bases}
	By our standing assumption, all manifolds in this note are Hausdorff.
	The conclusions of Theorem~\ref{th:main} and Proposition~\ref{prop:keycorollary}
	may fail without this assumption. For example, let
	\[
	M=\mathbb R^2\setminus\{(0,0)\}
	\]
	with the foliation given by connected components of the vertical lines. The leaf space $N$ is the real line with doubled origin, and  is a standard example of a non-Hausdorff manifold. The quotient map $\pi\colon M\to N$ is a surjective submersion, with contractible fibers.
	Nevertheless, $H^1_\pi(M)\neq 0$. A non-trivial class is represented by $\alpha\in \Omega^1_\pi(M)$, given by
	restriction of the angular form
	\[
	\omega=\frac{x \d y-y\d x}{x^2+y^2}\in\Omega^1(M)
	\]
	to the leaves. Suppose $\alpha=\d_\pi g$ for a function $g\in C^\infty(M)$. Then $\omega-\d g$ vanishes on $T_\pi M$, and hence is of the form
	\[ \omega-\d g=\phi(x,y)\d x\]
	for some function $\phi\in C^\infty(M)$. Let $\psi(x)=\phi(x,1)$.
	Since $\omega$ is closed, $\phi$ must be independent of $y$, and so $\phi(x,y)=\psi(x)$ for all $y$.
	(By continuity, this even holds true for the disconnected line $x=0$.)
	We find that $\omega=\d g+\psi(x)\d x$ is exact,  contradicting the fact that $\omega$ represents a~generator of $H^1(M)=\R$.
	
\subsection{The holomorphic setting}
	
	Let $A \Rightarrow N$ be a holomorphic Lie algebroid, with a
	representation on a holomorphic vector bundle $E \to N$. Following \cite[Section~6]{ev:tra},
	one defines the
	holomorphic Lie algebroid cohomology
	\[ H^\bullet_{\mathrm{hol}}(A,E)\]
	as the hypercohomology of the complex of sheaves of holomorphic sections
	\[ \mathcal O(\wedge^\bullet A^* \otimes E)\]
	with the Chevalley--Eilenberg differential.  It may alternatively be computed as the cohomology of the total complex of the
	Dolbeault double complex
	\[
	\Omega^{0,s}
	\bigl(
	N,\wedge^r A^* \otimes E
	\bigr),
	\]
	with differentials $\d_A$ and $\bar\partial$, of bidegrees $(1,0)$ and
	$(0,1)$. Laurent-Gengoux, Sti\'enon, and Xu \cite[Theorem 4.8]{lau:hol} showed that $T^{0,1}N$ and
	$A^{1,0}$ form a matched pair of complex Lie algebroids, and hence combine into a complex Lie algebroid
	\[
	A_{\mathrm{Dol}}
	=
	T^{0,1}N \bowtie A^{1,0}
	\]
	which we shall call the \emph{Dolbeault Lie algebroid}. The two Lie algebroid actions defining this matched pair are
	\[
	\nabla_v a=\bar\partial^A_v a,\qquad
	\nabla_a v=\pr^{0,1}[\varrho(a),v],
	\]
	for $v\in\Gamma(T^{0,1}N)$ and
	$a\in\Gamma(A^{1,0})$.
	Here $\bar\partial^A$ is the Dolbeault operator
	of the holomorphic vector bundle $A$, and
	$\varrho$ is the complexified anchor.
	
By \cite[Lemma~4.16]{lau:hol}, the holomorphic vector bundle $E$ carries a natural representation of~$A_{\mathrm{Dol}}$,
	\[
	\nabla^{\mathrm{Dol}}_{(v,a)}s
	=\bar\partial^E_v s+\nabla^E_a s,
	\]
	where $\nabla^E$ denotes the given holomorphic
	$A$-representation, extended to smooth sections.
	Its Chevalley--Eilenberg complex is naturally identified with the total
	Dolbeault complex above, and \cite[Theorem 4.19]{lau:hol} asserts that
	\begin{equation}\label{eq:gsx}
		H^\bullet_{\mathrm{hol}}(A,E)
		\cong
		H^\bullet(A_{\mathrm{Dol}},E).
	\end{equation}
	We next consider the functoriality of this construction.
	
	\begin{Lemma}\label{lem:functorial}
		Let $f\colon M\to N$ be a holomorphic map, transverse to the anchor
		of the holomorphic Lie algebroid $A\Rightarrow N$. Then there is a
		canonical isomorphism of complex Lie algebroids
		\[
		f^!(A_{\mathrm{Dol}})\cong \bigl(f^!A\bigr)_{\mathrm{Dol}},
		\]
		compatible with the induced representations on
		$f^*E$. The following diagram commutes:
		\[
		\xymatrix@C=4.5em@R=2.5em{
			H^\bullet_{\mathrm{hol}}(A,E)
			\ar[r]^{f^*}
			\ar[d]_{\cong}
			&
			H^\bullet_{\mathrm{hol}}\bigl(f^!A,f^*E\bigr)
			\ar[d]^{\cong}
			\\
			H^\bullet(A_{\mathrm{Dol}},E)
			\ar[r]_{f^*}
			&
			H^\bullet((f^!A)_{\mathrm{Dol}},f^*E).
		}
		\]
	\end{Lemma}
	\begin{proof}
Let
\[ 		(v,(w,a))\in  f^!(A_{\mathrm{Dol}})\subset T_\C M\times f^*\bigl(T^{0,1}N\oplus A^{1,0}\bigr).\]
Thus $T_{\mathbb C}f(v)=w+\varrho(a)$.
		Decomposing into holomorphic and anti-holomorphic parts, and using that
		$f$ is holomorphic, we have
		\[
		w=\bigl(T^{0,1}f\bigr)\bigl(v^{0,1}\bigr),
		\qquad
		\bigl(T^{1,0}f\bigr)\bigl(v^{1,0}\bigr)=\varrho(a).
		\]
		Consequently, the map taking $(v,(w,a))$ to $\bigl(v^{0,1},\bigl(v^{1,0},a\bigr)\bigr)$ defines an isomorphism of complex
		vector bundles
		\[
		f^!(A_{\mathrm{Dol}})
		\to
		T^{0,1}M\oplus\bigl(f^!A\bigr)^{1,0}
		=\bigl(f^!A\bigr)_{\mathrm{Dol}}.
		\]
		The definitions of the Dolbeault Lie algebroid structures show that this is an isomorphism of
		complex Lie algebroids. It also intertwines the induced
		representations on $f^*E$.
		
		Under the identification of the Chevalley--Eilenberg complexes of the
		Dolbeault Lie algebroids with the corresponding total Dolbeault
		complexes, the induced cochain maps are both given by pullback.
		The diagram therefore commutes.
	\end{proof}

Specializing to the case of submersions, we arrive at the following result.
	
\begin{Theorem}
Let $\pi\colon M\to N$ be a surjective holomorphic submersion with connected
fibers whose ordinary de~Rham cohomology vanishes in degrees
$1\leq k\leq n$.
Let $A\Ra N$ be a holomorphic Lie algebroid, with a representation on a holomorphic vector bundle $E\to N$. Then
the pullback map
\[ \pi^*\colon \ H^\bullet_{\mathrm{hol}}(A,E)\to H^\bullet_{\mathrm{hol}}\bigl(\pi^!A,\pi^*E\bigr)\]
is an isomorphism in degrees $0\le k\le n$, and is injective in degree $n+1$.
\end{Theorem}
	
\begin{proof}
Using the isomorphism \eqref{eq:gsx}, together with the functoriality statement from Lemma~\ref{lem:functorial}, this
		is equivalent to the statement that $\pi^*\colon H^\bullet(A_{\on{Dol}},E)\to  H^\bullet\bigl(\pi^!(A_{\on{Dol}}),\pi^*E\bigr)$
		is an isomorphism in degrees $0\le k\le n$ and injective in degree $n+1$. But this is just Theorem~\ref{th:main} for the case of complex Lie algebroids.
	\end{proof}
	
\subsection*{Acknowledgements}
	
I thank Marius Crainic for discussions on the sidelines of the June
	2025 Lisbon conference in honour of Rui Fernandes. I am also grateful to
	Maarten Mol for patient explanations of the Bernstein--Lunts argument
	and for pointing out Buchdahl's closely related result, which became
	the starting point for this note. Thanks also to the anonymous referees for careful reading of the manuscript and several valuable
	suggestions.

\pdfbookmark[1]{References}{ref}
\LastPageEnding
	
\end{document}